\documentclass[11pt,
reqno]{amsart}
\usepackage{array,latexsym}

\usepackage{amsmath,amssymb,amscd,amsfonts,amsthm}
\usepackage{mathrsfs}
\usepackage{dsfont}
\usepackage{bbold}
\usepackage{mathbbol}
\usepackage[all]{xy}
\usepackage{enumerate}

{\catcode`\@=11
\gdef\n@te#1#2{\leavevmode\vadjust{%
 {\setbox\z@\hbox to\z@{\strut#1}%
  \setbox\z@\hbox{\raise\dp\strutbox\box\z@}\ht\z@=\z@\dp\z@=\z@%
  #2\box\z@}}}
\gdef\leftnote#1{\n@te{\hss#1\quad}{}}
\gdef\rightnote#1{\n@te{\quad\kern-\leftskip#1\hss}{\moveright\hsize}}
\gdef\?{\FN@\qumark}
\gdef\qumark{\ifx\next"\DN@"##1"{\leftnote{\rm##1}}\else
 \DN@{\leftnote{\rm??}}\fi{\rm??}\next@}}
\DeclareFontFamily{OT1}{wncyr}{\hyphenchar\font45
}
\DeclareFontShape{OT1}{wncyr}{m}{n}{%
   <5> <6> <7> <8> <9> gen * wncyr
   <10> <10.95> <12> <14.4> <17.28> <20.74>  <24.88>wncyr10}{}
\DeclareFontShape{OT1}{wncyr}{m}{it}{%
   <5> <6> <7> <8> <9> gen * wncyi
   <10> <10.95> <12> <14.4> <17.28> <20.74> <24.88> wncyi10}{}
\DeclareFontShape{OT1}{wncyr}{m}{sc}{%
   <5> <6> <7> <8> <9> <10> <10.95> <12> <14.4>
   <17.28> <20.74> <24.88>wncysc10}{}
\DeclareFontShape{OT1}{wncyr}{b}{n}{%
   <5> <6> <7> <8> <9> gen * wncyb
   <10> <10.95> <12> <14.4> <17.28> <20.74> <24.88>wncyb10}{}
\input cyracc.def
\def\rus{\usefont{OT1}{wncyr}{m}{n}\cyracc\fontsize{9}{10pt}\selectfont}
\def\rusit{\usefont{OT1}{wncyr}{m}{it}\cyracc\fontsize{9}{10pt}\selectfont}

\DeclareMathSizes{9}{9}{7}{5} 

\theoremstyle{plain}

\newtheorem{theorem}{Theorem}[section]

\newtheorem{proposition}[theorem]{Proposition}
\newtheorem{lemma}[theorem]{Lemma}
\newtheorem{remark}[theorem]{Remark}
\newtheorem{remarks}[theorem]{Remarks}
\newtheorem{corollary}[theorem]{Corollary}

\theoremstyle{definition}

\newtheorem{definition}[theorem]{Definition}

\newtheorem{nothing*}[theorem]{}
\newtheorem{subnothing*}[sub]{}
\newtheorem{example}[theorem]{Example}

\theoremstyle{remark}

\def\g{{\mathfrak g}}

\def\a{{\mathfrak a}}
\def\b{{\mathfrak b}}
\def\z{{\mathfrak z}}
\def\t{{\mathfrak t}}
\def\s{{\mathfrak s}}

\def\c{{\mathfrak c}}

\def\cd{\hskip -.5mm\cdot\hskip -.5mm}

\newcommand{\ds}{{\raisebox{-.7\height}{$\cdot$}}\hskip -.65mm
{\raisebox{-.17\height}{$\cdot$}}\hskip -.65mm
{\raisebox{0.36\height}{$\cdot$}}\hskip -.65mm
{\raisebox{.9\height}{$\cdot$}}}

\begin{document}

\title[Modality of representations]{Modality of representations}

\author[Vladimir  L. Popov]{Vladimir  L. Popov${}^*$}
\address{Steklov Mathematical Institute,
Russian Academy of Sciences, Gubkina 8,
Moscow 119991, Russia}
\address{National Research University\\ Higher School of Economics\\ Myasnitskaya
20\\ Moscow 101000,\;Russia}
\email{popovvl@mi.ras.ru}

\thanks{
 ${}^*$\,This work was supported by the Russian Foundation
 for Basic Research, project no.\;15-01-02158.}

\dedicatory{To the memory of Bert Kostant}

\maketitle

\begin{abstract}
We first establish
several
general properties of modality of algebraic group actions.\;In particular, we introduce the notion of a modali\-ty-regular action and prove that every visible action is moda\-li\-ty-regular.
Then, using these results,
we classify irreducible linear representations of con\-nec\-ted simple algebraic groups of
every fixed modality $\leqslant 2$.\;Next, exploring
a finer geometric structure of linear actions, we generalize to the case of any cyclically graded semisimple Lie algebra the notion of a packet (or a Jordan/decomposition class)  and establish the properties of packets.
\end{abstract}

\section{\bf Introduction}

The modality of a group action is the maximal number of parameters
 on which a family of orbits may depend.\;This notion, as a natural measure of complexity of a group action, goes back to V.~I.~Arnold's works on the theory of singularities in which the actions of diffeomorphism groups
on the spaces of functions have been explored.\;V.~I.~Arnold and his collaborators succeeded in classifying the cases of a small modality (0 and 1); this led to the famous lists of singularities that enjoy remarkable properties; see \cite{A75}.

The concept of modality naturally adapts  to the
setting of algebraic group actions on algebraic varieties \cite{V86}, \cite[Sect.\;5.2]{PV94}.\;As in the theory of singularities, this concept allows one to systematically approach the classifica\-ti\-on problem of algebraic group actions by the degree of complexity. Here we explore the concept of modality in this setting.

In Section\;\ref{secmo} we first discuss general  properties of the modality of algebraic group actions; in particular, we introduce the notion of a modality-regular action
and prove that every visible action is modality-regular.\;Then we consi\-der the basic  class of algebraic group actions,
namely, that of the linear actions on (finite-dimensional) vector spaces.\;Guided by the analogy with V.~I.~Arnold's standpoint in the theory of singularities and with a view of obtaining the distinguished classes of linear actions,
we consider the problem of classifying linear actions (representations)
of a small modality and classify all irreducible representations of simple algebraic groups of every fixed modality $\leqslant 2$.\;They are, indeed, turn out to be remarkable because of allowing some nice equivalent characterizations.\;Actually, for modality $0$ the classification is not new:
the definition of modality implies that the class
of representations of modality $0$ coincides with that of representations with finitely many orbits; it has been given much attention to the latter in the literature, in particular, all  irreducible representations of reductive groups from this class have been listed \cite{SK77}, \cite{Ka80}, \cite{KKY86}.\;For simple groups, however, the nice characteriza\-ti\-ons are possible, which for all reductive groups are no longer satisfied (see below Remark \ref{rem}).

In Section\;\ref{secpac},
we explore the finer properties of the geometry of linear actions. Name\-ly,
a finer study of actions presupposes finding not only maximal number of parameters on which a family of orbits may depend (i.e., the modality),
 but also describing all maximal families of orbits (i.e., the sheets), and, where possible, presenting varieties as the disjoint unions of the finer families
which have a standard structure and the better controlled geometric properties.\;For the adjoint representations of semisimple algebraic groups, and the isotropy representations of symmetric spaces, the solutions to these finer study problems are known, see  \cite[Chap.\;39]{TY05} and the references therein.\;In these cases, the latter finer families are the Jordan classes (also known as the decomposition classes and the packets).\;We generalize this notion (using the term ``packet'') to the case of any
cyclically graded semisimple Lie algebra (or $\theta$-group, in the terminology of
\cite{V76}, \cite{Ka80}).\;We describe all packets and explore their properties; in particular, we  find their dimensions and modality.

 As the base field we fix an algebraically closed field $k$ of characteristic zero. Below we freely use the standard notation and terminology of algebraic group theory and invariant theory from
\cite{B91} and \cite{PV94}, where also the proofs of unreferenced claims and/or the relevant references can be found.\;All considered actions of algebraic groups on algebraic varieties and all homomorphisms of algebraic groups (in particular, representations) are assumed to be algebraic (i.e., regular/morphic).\;Topological terms are
related to the Zariski topology.

The results of Section\;\ref{secmo} are partly announced in \cite{P171}.

\section{\bf Modality}\label{secmo}

\subsection{}
 Let $G$ be a connected algebraic group.
We call any irreducible algebraic variety $F$ endowed with an action of $G$ such that all $G$-orbits in $F$ have the same dimension $d$ a {\it family} of $G$-orbits depending on
\begin{equation}\label{mo}
{\rm mod}(G:F):=\dim F-d
\end{equation}
parameters; the integer ${\rm mod}(G:F)$ is called the {\it modality} of $F$.\;If $F\dashrightarrow F\ds G$ is a rational quotient
of this action (which exists by the Rosenlicht  theorem), then
\begin{equation}\label{mtd}
{\rm mod}(G:F)=\dim F\ds G={\rm tr\,deg}_kk(F)^G
\end{equation}
and $F\ds G$ may be informally viewed as
the variety parametrizing typical $G$-orbits in\;$F$.

Given an  algebraic variety $X$ endowed with an action of $G$, we denote by ${\mathscr F}(X)$ the set of all locally closed $G$-stable subsets of $X$ which are families. The integer
 \begin{equation}\label{mod}
 {\rm mod}(G:X):=\underset{F\in {\mathscr F}(X)}{\rm max}{\rm mod}(G:F),
 \end{equation}
 is then called the {\it modality} of $X$.\;If $X$ is a vector space and the action  is linear determined by a representation $\varrho\colon G\to {\rm GL}(X)$, then we call ${\rm mod}(G:X)$ the {\it modality of representation} $\varrho$ and denote it by ${\rm mod}\,\varrho$.

 If $Y$ is an algebraic variety endowed with an action of a
  (not necessarily connected) algebraic group $H$ and if $H^0$ is the identity component of $H$, then, by defi\-ni\-ti\-on\footnote{This definition fixes the inaccuracy in  \cite{V86}, \cite[Sect.\;5.2]{PV94}, where ${\rm mod}(G:X)$ is defined by \eqref{mod} for any $G$, not necessarily connected:
 as is easily seen,  for a disconnected $G$, the set ${\mathscr F}(X)$ may be empty,
 so this definition should be corrected.},
 \begin{equation*}\label{ccomp}
 {\rm mod}(H:Y):={\rm mod}(H^0:Y).
 \end{equation*}
 Similarly, the modality of a representation of $H$ is de\-fined as the modality of its restriction to $H^0$.

 Recall that, for every integer $d$, the set $\{y\in Y\mid \dim H\cd y\leqslant d\}$ is closed in $Y$.\;Whence, for every locally closed irreducible (not necessarily $H$-stable) subset $Z$ in $Y$,
 the subset
 \begin{equation}\label{regreg}
 Z^{\rm reg}:=\{z\in Z\mid \dim H\cd z\geqslant \dim H\cd y\;\mbox{for every $y\in Z$}\}
 \end{equation}
is dense and open in $Z$.

 The definition of modality implies that equality \eqref{mod} still holds
 if ${\mathscr F}(X)$ is replaced by the set of all  maximal (with respect to inclusion) families in $X$, i.e., by the {\it sheets} of $X$ \cite[Sect.\;6.10]{PV94}.\;Recall
 that there are only finitely many sheets of $X$. If $X$ is irreducible,
 then
 $X^{\rm reg}$ is a sheet, called {\it regular}, which is open and dense in\;$X$.
 It follows from \eqref{mtd}
 that
 \begin{equation}\label{mtr}
 {\rm mod}(G:X^{\rm reg})={\rm tr}\,{\rm deg}_k k(X)^G.
 \end{equation}
This implies  that equality \eqref{mod} still holds if  ${\mathscr F}(X)$ is replaced by the set of all
$G$-stable locally closed (or closed) subsets of $X$, and ${\rm mod}(G:F)$ by
${\rm tr\,deg}_kk(F)^G$.

The aforesaid shows that ${\rm mod}(G:X)=0$ if and only if the set of all $G$-orbits in $X$ is finite.

If $G$ is reductive and $X$ is affine, then
\begin{equation}\label{redmo}
\begin{split}
{\rm mod}(G:X)&\geqslant {\rm mod}(G:X^{\rm reg})\overset{\eqref{mtr}}{=\hskip -2.0mm=}{\rm tr}\,{\rm deg}_k k(X)^G\\
&\geqslant {\rm tr}\,{\rm deg}_k k[X]^G=
\dim X/\!\!/G.
\end{split}
\end{equation}

\subsection{} The existence of regular sheets leads to defining the following naturally distinguished class of actions:

  \begin{definition}\label{reg} An action of a connected algebraic group $G$ on an irreducible algebraic variety $X$ is called {\it modality-regular} if ${\rm mod}(G:X)\!=\!{\rm mod}(G:X^{\rm reg})$.\;A linear representation
  $G\to {\rm GL}(V)$ is called  {\it modality-regular} if it determines
  a modality-regular action of $G$ on $V$.
 \end{definition}

There are affine algebraic groups $G$ such that every action of $G$ is modality-regular.
Here is their complete classification:

 \begin{theorem}[{{\rm groups all actions of which are modality-regular}}]\label{mcri} The follow\-ing properties of a connected affine algebraic group $G$ are equivalent:
 \begin{enumerate}[\hskip 4.2mm\rm(i)]
 \item all actions of $\,G$ on irreducible algebraic varieties are modality-regular;
 \item
 for every irreducible algebraic variety $X$ and every action of $\,G$ on $X$ with a dense open $G$-orbit, there are only finitely many $\,G$-orbits in $X$;
 \item
 $G$ is
 one of the following groups:
 \begin{enumerate}[\hskip 0mm---]
\item a torus,
 \item a product of a torus and a group isomorphic to
 ${\bf G}_a$.
 \end{enumerate}
 \end{enumerate}
 \end{theorem}

 \begin{proof} See \cite{P172}.
\end{proof}

\subsection{} However, a restriction of the class of actions under consideration may lead to an extension of the class of those groups for which all actions of this class are modality-regular.\;Apparently
 for the first time such a phenomenon was discovered
in the following theorem:

 \begin{theorem}
 \label{BU}
 Let $G$, $B$, and $U$ be respectively a connected reductive algebraic group, a Borel subgroup of $\,G$, and a maximal unipotent subgroup of $G$.
 Then the restrictions to $B$ and $U$ of any action of $\,G$ on an irreducible algebraic variety
 are modality-regular.
 \end{theorem}

 \begin{proof} See \cite[Thms. 2 and 3]{V86}.
 \end{proof}

\begin{remark} {\rm By Theorem \ref{mcri}, if $G$ is not abelian, then there are  actions of $B$ which are not modality-regular.\;By Theorem \ref{BU}, these $B$-actions can not be extended up to $G$-actions.
If ${\rm rk}\,G\geqslant 2$, the same holds for $U$.}
\end{remark}

\subsection{} The next example of this phenomenon, in particular, shows that, apart from tori, there are other connected reductive algebraic groups for which every representation is modality-regular:

 \begin{theorem}[{{actions of  ${\rm SL}_2$}}]\label{SL2} Every action of $\,G={\rm SL}_2$ on an irreducible quasiaffine algebraic va\-rie\-ty $X$ is modality-regular.
 \end{theorem}
 \begin{proof} Since $X$ is quasiaffine, there is an equivariant open embedding of $X$ in an affine algebraic variety
 endowed with an action of $G$; see \cite[Thm.\;1.6]{PV94}. Therefore,  we may (and shall) assume that $X$ is affine.
 Given a  sheet $S\neq X^{\rm reg}$,  we need to show that
 \begin{equation}\label{ineq}
 {\rm mod}(G:S)\leqslant {\rm mod}(G:X^{\rm reg}).
 \end{equation}

First we note that
\begin{equation}\label{fd}
\dim X^G\leqslant {\rm mod}(G:X^{\rm reg}).
\end{equation}
Indeed, since $k[X]^G$ separates closed orbits, the restriction  of the quotient morphism
$X\to X/\!\!/G$  to $X^G$ is injective. Hence $\dim X^G\leqslant \dim X/\!\!/G$. This and
\eqref{mtr} imply \eqref{fd} because
$\dim X/\!\!/G={\rm tr\,deg}_kk[X]\leqslant {\rm tr}\,{\rm deg}_k k(X)^G$.

In view of \eqref{fd} we need to consider only the case where $S \cap X^G=\varnothing$. Assume that this equality holds.\;As is well-known, every one-dimensional homogeneous space of $G$ is projective (actually, isomorphic to ${\bf P}^1$).\;Whence, since every  $G$-orbit in $X$ is quasiaffine, its dimension may be only $\dim G=3$, $2$, or $0$.\;This
and $S\neq X^{\rm reg}$ imply that
 \begin{equation}\label{sd}
\dim \,G\cd x=\begin{cases} 2 &\mbox{if $x\in S$},\\
 3 &\mbox{if $x\in X^{\rm reg}$,}
 \end{cases}
 \quad\mbox{and}\quad
 \dim S\leqslant \dim X-1.
 \end{equation}
From \eqref{mo},\;\eqref{sd} we get ${\rm mod}(G:S)\!\leqslant\! \dim X\!-\!3\!=\!{\rm mod}(G:X^{\rm reg})$, whence \eqref{ineq}.
 \end{proof}

\begin{remark} {\rm By Theorem \ref{mcri}, the quasiaffinity condition in Theorem \ref{SL2} can not be dropped.}
\end{remark}

\subsection{} Using Theorem \ref{SL2}, one computes the modality of  every representation of ${\rm SL}_2$. Namely,
let  $\varrho_n$ be the
$(n+1)$-dimensional
linear representation of ${\rm SL}_2$
determining the natural ${\rm SL}_2$-module structure on the space
of binary forms of degree $n$ over $k$.\;It
is irreducible and
every linear representation of ${\rm SL}_2$ is equivalent to a direct sum of
such representations. Given a representation $\varrho$ and an integer $s>0$, we denote $s\varrho:=\varrho\oplus\cdots\oplus\varrho$ ($s$ summands).
 \begin{corollary}[{{\rm modality of ${\rm SL}_2$-representations}}] Let
 $\varrho\colon G\to {\rm GL}(V)$ be a linear representation
 of $\,G={\rm SL}_2$.\;Then
 \begin{equation*}
 {\rm mod} \,\varrho =\begin{cases}
 \dim \varrho-3&\mbox{if $\varrho\neq s\varrho_0\oplus \varrho_1,\, s\varrho_0\oplus \varrho_2,\,s\varrho_0$},\\
 \dim \varrho-2& \mbox{if $\varrho = s\varrho_0\oplus \varrho_1,
 s\varrho_0\oplus \varrho_2$},\\
  \dim\varrho & \mbox{if $\varrho = s\varrho_0$}.
 \end{cases}
 \end{equation*}
 \end{corollary}

\begin{proof} By Definition \ref{reg} and Theorem \ref{SL2} we have ${\rm mod}\,\varrho=\dim\varrho-\underset{v\in V}{\max}\,\dim G\cdot v$.\;The integer $\underset{v\in V}{\max}\,\dim G\cdot v$ is computed \cite{P74}. Whence the claim.
\end{proof}

\subsection{}  The following example shows that for every integer $n\geqslant 3$ there are linear representations of $G={\rm SL}_n$ which are not modality-regular.\;In particular,
${\rm SL}_2$ in Theorem \ref{SL2} cannot be replaced  by ${\rm SL}_n$ for  $n\geqslant 3$.

\begin{example}\label{exmo}
Consider the natural action of $G$ on $k^n$ and the diagonal action of $\,G$ on $V\!:=\!k^n\oplus\cdots\oplus k^n$ ($d$ summands).
If $d\leqslant n-1$, there is an open $G$-orbit in $V$, so we have ${\rm mod}(G:V^{\rm reg})=0$.\;On the other hand, for every nonzero
$v, u\in k^n$ and $\lambda_1,\ldots, \lambda_{d-1}, \mu_1,\ldots, \mu_{d-1}\in k^\times$ the elements
$(v,\lambda_1 v,\ldots, \lambda_{d-1} v)\in V$ and $(u,\mu_1 u,\ldots, \mu_{d-1} u)\in V$ lie in the same $G$-orbit if and only if $\lambda_i=\mu_i$ for all $i$.\;This imples that
${\rm mod}(G:V)\geqslant d-1$.
\end{example}

\subsection{} We shall now prove that all the representations from a certain important class are modality-regular; this will be then used in the proof of the classificati\-on results in Subsection \ref{clcl}.

Recall from \cite{Ka75}, \cite{Ka80} (see also \cite[\S8]{PV94}) that a linear action of a reductive algebraic group $G$ on a vector space $V$ (and the corresponding representation $G\to {\rm GL}(V)$) is called {\it visible} if there are only finitely many $G$-orbits in the level variety of $k[V]^G$ in $V$ containing $0$.\;As a matter of fact,
then automatically  every level variety of $k[V]^G$ in $V$ contains only finitely many $G$-orbits \cite[Cor.\;3 of Prop.\;5.1]{PV94}.\;Extending this terminology to a more general setting, we introduce the following

 \begin{definition}
 An action of $G$ on an affine algebraic variety $X$ is called {\it visible} if every fiber of the categorical quotient
\begin{equation}\label{ct}
\pi_{G, X}^{\ }\colon X\to X/\!\!/G
\end{equation}
contains only finitely many $G$-orbits.
\end{definition}

\begin{theorem}[{{\rm modality of visible actions}}] \label{visact} Every visible action
of a reduc\-ti\-ve algebraic group $G$ on an irreducible affine algebraic variety $X$
enjoys the following properties:
\begin{enumerate}[\hskip 4.2mm \rm(i)]
    \item it is modality-regular;
    \item ${\rm mod}(G:X)=\dim X/\!\!/G$;
    \item the induced action of $\,G$ on every
    closed $G$-stable subset of $X$ is visible.
\end{enumerate}
\end{theorem}
\begin{proof} First, we prove (iii).\;Let
$Y$ be a $G$-stable closed subset of $X$.\;Then the set $Z:=\pi_{G, X}^{\ }(Y)$ is closed in $X/\!\!/G$ and $\pi_{G, X}^{\ }|_{Y}\colon Y\to Z$ is the categorical quotient for the action of $G$ on $Y$; see \cite[Sect.\;4.4]{PV94}.\;Every fiber of $\pi_{G, X}^{\ }|_{Y}$ is the intersection of  $Y$ with a fiber of $\pi_{G, X}^{\ }$; since the latter
contains only finitely-many $G$-orbits, this intersection shares this property.

Now we prove (i) and (ii). Let $m_{G, X}^{\ }:=\max_{x\in X}\dim G\cdot x$.\;For every $x\in X^{\rm reg}$ we then have $\dim G\cd x=m_{G, X}^{\ }$.\;Since the fiber $\pi_{G, X}^{-1}(\pi_{G, X}^{\ }(x))$ contains only finitely many orbits, the latter equality entails that the dimension of this fiber is equal to $m_{G, X}^{\ }$.\;From this we infer that
\begin{equation}\label{mgx}
{\rm mod}(G:X^{\rm reg})=\dim X^{\rm reg}-m_{G, X}^{\ }=\dim X-m_{G, X}^{\ }=\dim X/\!\!/G
\end{equation}
(the second equality in \eqref{mgx} holds as $X^{\rm reg}$ is open in $X$,
and the third by the fiber dimension theorem).

Let $S$ be a sheet of $X$, let $\overline S$ be its closure in $X$, and let $Z:=\pi_{G, X}({\overline S})={\overline S}/\!\!/G$.\;As
$\overline S$ is a $G$-stable closed subset in $X$,
the action of $G$ on $\overline S$ is visible by (iii).\;Therefore,
replacing $X$ in \eqref{mgx} by ${\overline S}$ and taking into account that
$S={\overline S}^{\rm reg}$,  we obtain
\begin{equation}\label{zx}
{\rm mod}(G:S)=\dim Z.
\end{equation}

Now the inclusion $Z\subseteq X/\!\!/G$ combined with \eqref{mgx} and \eqref{zx}, yields the inequality ${\rm mod}(G:X^{\rm reg})\geqslant {\rm mod}(G:S)$.\;This completes the proof.
\end{proof}

\subsection{}\label{thetagroups}
Theorem \ref{visact} is applicable to the class
of so-called
$\theta$-groups studied in \cite{V76} (see also \cite{Ka80}).

Namely, let $m$ be either a positive integer or $\infty$.\;Denote by
${\mathbf Z}_m$ the following additively written cyclic group of order $m$:
for $m<\infty$,
the elements of
$\mathbf Z_m$
are the integers between $0$ and $m-1$,
and
the sum of $i$ and $j$ in $\mathbf Z_m$
is the remainder of dividing $i+j$ by $m$;
by definition,  $\mathbf Z_\infty=\mathbf Z$.

Consider a simply connected semisimple algebraic group $G$ and assume that its Lie algebra $\g={\rm Lie}\,G$  is $\mathbf Z_m$-graded:
\begin{equation}\label{gradC}
\g=\bigoplus_{i\in \mathbf Z_m}\g_i.
\end{equation}
Then ${\rm Aut}\,\g$ contains a subgroup $\theta$, which, for $m<\infty$,  is cyclic of order $m$,
and, for $m=\infty$, a one-dimensional torus,
such that  \eqref{gradC} is the weight
decomposition of $\g$ with respect to the natural action of $\theta$.

The component $\g_0$ in \eqref{gradC} is a reductive subalgebra of $\g$.\;Let $G_0$ be a closed connected subgroup of $G$ with
\begin{equation}\label{G0}
{\rm Lie}\,G_0=\g_0.
\end{equation}

Every $\g_i$ in \eqref{gradC} is
$G_0$-stable with respect to the adjoint action, so one can consider $\g_i$  as the $G_0$-module with respect to this action.\;As any $G_0$-module $\g_i$ coincides with the $G_0$-module $\g_1$ for another appropriate cyclic grading of $\g$, when studying the orbital decompositions
it suffices to explore only the $G_0$-module $\g_1$.\;By \cite[Thm.\;1]{V76},
any two maximal linear subspaces of  $\g_1$
consisting of semisimple paiwise commuting elements are transformed to each other by $G_0$.\;These subspaces are called the {\it Cartan subspaces} and their mutual dimension is called the {\it rank of the graded Lie algebra} \eqref{gradC}.\;The image of
the adjoint representation $G_0\to {\rm GL}(\g_1)$ is called the {\it $\theta$-group} associated with \eqref{gradC}.

\begin{theorem}[{{\rm modality of $\theta$-group actions}}]\label{mr} In the above notation,
the adjoint representation $\varrho\colon G_0\to {\rm GL}(\g_1)$ is modality-re\-gu\-lar and
 its modality is equal to
the rank $r$
of the graded Lie algebra {\rm \eqref{gradC}}.
\end{theorem}

\begin{proof} By \cite[Thm.\;4]{V76}, $\varrho$ is visible, and, by \cite[Thm.\;5]{V76}, $\dim \g_1/\!\!/G_0=r$.\;In view of this, the claim follows from Theorem
\ref{visact}.
\end{proof}

\begin{corollary}[{{\rm case $m=1$}}]\label{c1} The adjoint representation of every connected semisimple algebraic group $G$
is modality-regular and its modality is equal to  the rank of $\,G$.
\end{corollary}

\begin{corollary}[{{\rm case $m=2$}}]\label{c2} The isotropy representation of every symmetric space $X$ of a connected semisimple algebraic group is modality-regular and its modlity is equal to the rank of $X$.
\end{corollary}

\subsection{} We now turn to classifying representations of a small modality.\;First, note that the following finiteness theorem can be considered as an argument in favor of considering the problem of classifying representations in terms of the magnitude of the modality.

\begin{theorem}[{{\rm finiteness for modality}}] For every  connected simisimple algebraic group $G$ and every integer $m\geqslant 0$, there are only
finitely many {\rm (}up to equivalence\,{\rm )} linear representations of $\,G$ of modality  $m$.
\end{theorem}
\begin{proof} First note that $G$ has only finitely many  {\rm (}up to equivalence{\rm )}
linear representations of any fixed dimension.\;Indeed, given the complete reducibility
of representations, it suffices to prove this for irreducible representations. Denote by
$\varrho(\lambda)$ the irreducible representation of $G$ with the highest weight $\lambda$
regarding a fixed Borel subgroup and its torus $T$.\;We may (and shall) assume that $G$ is simply connected; let then $\varpi_1,\ldots,\varpi_r$ be the fundamental weights of $T$.\;It follows from the Weyl formula for $\dim \varrho(\lambda)$, see
\cite[Chap.\;VIII, \S4, (41)]{J62},  that
$\dim \varrho(\sum_{i=1}^r n_i\varpi_i)
<\dim \varrho(\sum_{i=1}^r m_i\varpi_i
)$ if $0 \leqslant n_i\leqslant m_i$ for all $i$ and
$n_{i_0}< m_{i_0}$ for some\;$i_0$.\;This implies the
finiteness statement.

The claim of the theorem now follows from this statement combined with the inequality
${\rm mod}\,\varrho\leqslant \dim\varrho$ that stems from
the definition of ${\rm mod}\,\varrho$.
\end{proof}

\subsection{}\label{clcl}
In Theorem \ref{cclassification} below, we classify irreducible representations
of connec\-ted simple algebraic groups
of modalities $0, 1$, and $2$.\;We use in this theorem  the following agreements and notation.

 Let $G$ be a a connected semisimple algebraic group and let
 $\pi\colon \widetilde G\to G$ be its universal covering.\;The map  $\varrho\mapsto
 \widetilde \varrho:=\varrho\circ\pi$ is
 a bijection between the set of all representations of
 $G$ and the set of all representations of $\widetilde G$ factoring through $G$.\;This allows one to specify $\varrho$ by specifying $\widetilde \varrho$.\;If  $\widetilde \varrho$ is irreducible, it is uniquely up to equivalence determined by its highest weight  $\lambda$ (with respect to a fixed Borel subgroup $B$ of $\widetilde G$ and its maximal torus $T$).\;Given this, we denote $\varrho$ (considered up to equivalence) by $({\sf R}, \lambda)$, where ${\sf R}$ is the type of the root system of $\widetilde G$.\;The fundamental weights of $\widetilde G$ with respect to the pair $(B, T)$ are denoted by $\varpi_1, \ldots, \varpi_r$; we use their Bourbaki numbering \cite{Bou68}.\;For ${\sf R}={\sf A}_r, {\sf B}_r, {\sf C}_r, {\sf D}_r$, we assume that, respectively, $r\geqslant 1, 3, 2, 4$.\;The group of characters of $T$ is considered in additive notation.\;The representation contragredient to
$\varrho$ is denoted by  $\varrho^*$.

\begin{theorem}[{{\rm
irreducible representations
of simple algebraic groups
of modalities
$0, 1, 2$}}]\label{cclassification}\hskip -1mm Let $G$ be a connected simple algebraic group and let $\,V$\;be a $G$-module determined by a nontrivial irreducible representation $\varrho\colon G\!\to\! {\rm GL}(V)$.
\begin{enumerate}[\hskip 4.2mm ${\rm(M_1)}$]
\item[${\rm(M_1)}$]
The condition ${\rm mod}\,\varrho=0$ is equivalent to either of the following{\rm :}
\begin{enumerate}[\hskip 0mm \rm(i)]
\item $k[V]^G=k$;
\item the action of $\,G$ on $V$ is nonstable;
\item $\varrho$ or $\varrho^*$ is contained in the following list:
 $$({\sf A}_r, \varpi_1); ({\sf A}_r, \varpi_2), r\geqslant 4
\mbox{ \it is even\,}; ({\sf C}_r, \varpi_1);
({\sf D}_5, \varpi_5).$$
\end{enumerate}
\item[${\rm(M_2)}$] The condition ${\rm mod}\,\varrho=1$ is equivalent to either of the following{\rm :}
\begin{enumerate}[\hskip 0mm \rm(i)]
\item ${\rm tr\,deg}_k k[V]^G=1$;
\item $V/\!\!/G={\mathbf A}^1$;
\item $\varrho$ or $\varrho^*$ is contained in the following list:
\begin{gather*}
({\sf A}_r, 2\varpi_1);
({\sf A}_r, \varpi_2), r\geqslant 3 \mbox{ \it is odd\,};
({\sf B}_r, \varpi_1); ({\sf D}_r, \varpi_1); \\
({\sf A}_1, 3\varpi_1); ({\sf A}_5, \varpi_3); ({\sf A}_6, \varpi_3);  ({\sf A}_7, \varpi_3); ({\sf B}_3, \varpi_3);
({\sf B}_4, \varpi_4);
 ({\sf B}_5, \varpi_5);\\
 ({\sf C}_2, \varpi_2); ({\sf C}_3, \varpi_3); ({\sf D}_6, \varpi_6); ({\sf D}_7, \varpi_7);
({\sf G}_2, \varpi_1); ({\sf E}_6, \varpi_1);
({\sf E}_7, \varpi_7).
\end{gather*}
\end{enumerate}
\item[${\rm(M_3)}$] The condition ${\rm mod}\,\varrho=2$ is equivalent to either of the following {\rm :}
\begin{enumerate}[\hskip 0mm \rm(i)]
\item ${\rm tr\,deg}_k k[V]^G=2$;
\item $V/\!\!/G={\mathbf A}^2$;
\item $\varrho$ or $\varrho^*$ is contained in the following list:
\begin{gather*}
 ({\sf A}_1, 4\varpi_1); ({\sf A}_2, \varpi_1+\varpi_2); ({\sf A}_2, 3\varpi_1);  ({\sf B}_6, \varpi_6);\\
 ({\sf C}_2, 2\varpi_1);
({\sf C}_3, \varpi_2);
({\sf F}_4, \varpi_4); ({\sf G}_2, \varpi_2).
\end{gather*}
\end{enumerate}
\end{enumerate}

If ${\rm mod}\,\varrho\leqslant 2$, then $\varrho$ is modality-regular.
\end{theorem}

First we prove the following

\begin{lemma}\label{freealg} Let $G$ be a connected semisimple algebraic group and let $\,V$ be a $G$-module such that ${\rm mod}(G:V)\leqslant 2$. Then $k[V]^G$ is a free $k$-algebra.
\end{lemma}

\begin{proof} By \eqref{redmo}, we have  $\dim V/\!\!/G\leqslant 2$.\;If
$\dim V/\!\!/G=0$, then $k[V]^G=k$.\;If $\dim V/\!\!/G=1$, then it follows from the L\"uroth theorem that $V/\!\!/G={\bf A}^1$; see \cite[Prop.\;12]{P80} (this is true for any reductive $G$).\;If $\dim V/\!\!/G=2$, then, as is proved in \cite{K80} (and first conjectured in \cite[Sect.\;7, Rem.\;$2^\circ$]{P77}),
the assumption that
$G$ is connected semisimple entails $V/\!\!/G={\bf A}^2$ (cf.\;also \cite[Sect.\;8.4]{PV94}).\;This completes the proof.
\end{proof}

\begin{proof}[Proof of Theorem  {\rm \ref{cclassification}}]
By Lemma \ref{freealg}, if ${\rm mod}\,\varrho\leqslant 2$, then
$\varrho$ is cofree (i.e., $k[V]^G$ is a free $k$-al\-geb\-ra).\;The
list of all cofree irredu\-cible representations of connected simple algebraic group
is obtained in \cite{KPV76} (see also Summary Table in \cite[pp.\;259--262]{PV94}).\;This leads to determining which representations from this list have modality $\leqslant 2$.\;Comparing this list with the list
of all irreducible visible representations obtained in   \cite[Thm.\;1]{Ka80} shows that
every representation $\varrho$ from the former list is visible.\;Whence, by Theorem \ref{visact}, it is modality-regular and
${\rm mod}\,\varrho=\dim V/\!\!/G$. Since the integers $\dim V/\!\!/G$ are
known (they are specified in the fifth column of the Summary Table in \cite[pp.\;259--262]{PV94}), this yields the lists in
${\rm(M_1)}$(iii), ${\rm(M_2)}$(iii), and ${\rm(M_3)}$(iii), thereby proving all the claims except ${\rm(M_1)}$(ii).\;By \cite[Thm.\;1]{P71},
$\varrho$ is nonstable if and only if the $G$-stabilizer of a point in general position in $V$ is nonreductive.\;Since the stabilizers of points in general position for the representations from this list are known  as well (they are specified in the fourth column of the Summary Table in \cite[pp.\;259--262]{PV94}), applying this criterion yields that
${\rm(M_1)}$(ii) and
${\rm(M_1)}$(iii) are equivalent.\;This completes the proof.
\end{proof}

\begin{remarks}\label{rem}{\rm \

 1. In the following statements, the assumptions of the irreducibility of $\rho$ and the simplicity of $G$ are essential:

(a)  In ${\rm(M_1)}$,  in the claims that $k[V]=k$ or the nonstability of the action of $G$ on $V$
implies ${\rm mod}\,\varrho=0$.\;Examples:

---\,If $\rho=2({\sf A}_r, \varpi_1)$ and $r\geqslant 2$, then $k[V]^G=k$, the action of $G$ on $V$ is nonstable, and
${\rm mod}\,\rho>0$; see Example \ref{exmo}.

---\,If
$\varrho=({\sf A}_r, \varpi_2)\otimes ({\sf A}_1, \varpi_1)$, where $r\geqslant 8$ is even, then $k[V]^G=k$, see \cite{L89},
and the action of $G$ on $V$ in nonstable, see  \cite[Thm.\;3.3, Cor.\;of Thm.\;2.3]{PV94}, but the number of $G$-orbits in $V$ is infinite (equivalently, ${\rm mod}\,\varrho>0$);
see \cite{SK77}, \cite{K80}.

(b) In ${\rm(M_2)}$,  in the claim that ${\rm tr\,deg}_kk[V]^G=1$ is equivalent to ${\rm mod}\,\varrho=1$.\;For example, if $\rho=(r+1)({\sf A}_r, \varpi_1)$ and $r\geqslant 2$, then
${\rm tr\,deg}_k k[V]^G=1$, and ${\rm mod}\,\rho \geqslant r$; see \;Example \ref{exmo}.

(c) In ${\rm(M_3)}$, in the claim that ${\rm mod}\,\varrho=2$ implies ${\rm tr\,deg}_kk[V]^G=2$.\;For example, if $\rho=3({\sf A}_r, \varpi_1)$ and $r\geqslant 3$, then it is not difficult to see that $k[V]^G=k$ and ${\rm mod}\,\rho=2$.

The representations $\rho$ specified in (a), (b), (c) are not modality-regular.

\vskip 1mm

2.\;Arguing along the same lines, one can extend the classifications obtained in Theorem \ref{cclassification} up to the classifications of all irreducible castling reduced repre\-sentations
of connected semisimple algebraic groups of every fixed moda\-li\-ty $\leqslant 2$.\;The reason being
that
the complete list of
cofree  irreducible castling reduced representations
of connected semisimple algebraic groups is known \cite{L89}. Com\-paring it with that of visible representations obtained in \cite{SK77}, \cite{Ka80}, \cite{KKY86}, one ascertains that  the majority of representations in this list (all but two) are visible, therefore, for them, the modality is given by Theorem \ref{visact} (the integers $\dim V/\!\!/G$ are specified in  the fifth column of \cite[Tabelle]{L89}).\;To the remaining two representations one applies the ad hoc considerations.

\vskip 1mm

3.\;The arguments from the proof of
 Theorem  {\rm \ref{cclassification}}
 yield
the modalities of all cofree irredu\-cible represen\-tations of connected simple algebraic groups.
}
\end{remarks}

\subsection{} We conclude this section with a statement,
which in some cases helps to practically determine the modality.

\begin{lemma}\label{SI} Let $X$ be an algebraic variety endowed with an action of an algebraic group $G$.\;Let $\{C_i\}_{i\in I}$ be a collection of subsets of $X$ such that
\begin{enumerate}[\hskip 2.0mm\rm(i)]
\item $I$ is finite;
\item $\bigcup_{i\in I}C_i=X$;
\item the closure $\overline{C_i}$ of $\,C_i$ in $X$ is irreducible for every $i\in I$;
\item every $C_i$ is $G$-stable;
\item all $G$-orbits in $C_i$ have the same dimension $d_i$ for every $i\in I$.
\end{enumerate}
Then the following hold:
\begin{enumerate}[\hskip 4.2mm\rm(a)]
\item ${\rm mod}(G:X)=
{\rm max}_{i\in I}\big(\!\dim \overline{C_i}-d_i\big)$;
    \item if $X$ is irreducible, then $X=\overline{C_{i_0}}$ for some $i_0$,
    and ${\rm mod}(G:X^{\rm reg})=\dim X-d_{i_0}$.
\end{enumerate}
\end{lemma}

\begin{proof}

By (iii), we have a family ${\overline {C_i}}^{\rm reg}$, and (v) implies
$C_i\subseteq {\overline {C_i}}^{\rm reg}$.
Whence
\begin{equation}\label{Cm}
{\rm mod}(G:\overline {C_i}^{\rm reg})=\dim \overline {C_i}-d_i.
\end{equation}
From \eqref{mod} and \eqref{Cm}, we infer that
${\rm mod}(G:X)\geqslant \underset{i\in I}{\max} (\dim {\overline C}_i-d_i)$.\;To prove the op\-po\-site inequality let $Z\in {\mathscr F}(X)$ be a family of $s$-dimensional $G$-orbits such that ${\rm mod}(G:X)=\dim Z-s$
and let
$J:=\{i\in I \mid Z\cap C_i\neq \varnothing\}.$
 By (ii), we have $Z=\bigcup_{j\in J}(Z\cap {\overline {C_j}})$.\;Since $Z$ is irreducible and,
  by (i),  $J$ is finite, there is $j_0\in J$ such
 that $Z\subseteq {\overline {C_{j_0}}}$.\;As $Z\cap C_{j_0}\neq \varnothing$,
 we have $s=d_{j_0}$.\;Therefore,
 ${\rm mod}(G:X)=\dim Z-s\leqslant \dim
 \overline {C_{j_0}}-d_{j_0}$.\;This
 proves (a).

By (ii), $\bigcup_{i\in I} \overline{C_i}=X$.\;If $X$ is irreducible, then, in view of (i), this equality implies the existence of $i_0$ such that $X=\overline{C_{i_0}}$.\;This and \eqref{Cm}
prove (b).
\end{proof}

\begin{example} Consider the graded semisimple Lie algebra \eqref{gradC} and the adjo\-int representation $\varrho\colon G_0\to {\rm GL}(\g_1)$.\;In Section \ref{secpac} we define a collection\;$\{C_i\}_{i\in I}$ of subsets of $\g_1$ called packets (see below Definition \ref{packet}).\;In Proposi\-tions \ref{cover}, \ref{pf},  \ref{propac} below we prove that this collection satisfies all conditions (i)--(v) from Lemma \ref{SI}, and in Corollary \ref{mopa} below we compute ${\rm mod}(G_0:\overline {C_i})$ for every $i$.\;The integer
$\max_{i\in I}{\rm mod}(G_0:\overline {C_i})$ turns then out to be equal to
the rank of the graded Lie algebra \eqref{gradC}.\;By Lemma \ref{SI}, this agrees with
Theorem \ref{mr}.
\end{example}

\begin{corollary}\label{times-fixed}
 Let $\,Y$ and $F$ be the algebraic varieties endowed with the actions of an algebraic group $G$.\;Consider the diagonal action of $\,G$ on $X:=Y\times F$.\;If $\,G$ acts on $F$ trivially, then
 \begin{enumerate}[\hskip 4.2mm\rm(i)]
 \item ${\rm mod}(G:X)={\rm mod}(G:Y)+\dim F$;
 \item the action of $\,G$ on $Y$ is modality-regular if and only if the action of $\,G$ on $X$ is modality-regular.
 \end{enumerate}
\end{corollary}

\begin{proof} Let $\{S_i\}_{i\in I}$ be the collection of all sheets of $Y$.\;The claim then follows from Lemma \ref{SI} applied to $\{C_i\}_{i\in I}$  where $C_i:=S_i\times F$.
\end{proof}

\section{\bf Packets in cyclically graded semisimple Lie algebras}\label{secpac}

 In this section
we explore the finer geometric properties of
$\theta$-group actions.

\subsection{}
Besides the notation of  Subsection \ref{thetagroups} below we also use the following:

\begin{enumerate}[\hskip 4.2mm 
$\raisebox{.4\height}{\mbox{$\centerdot$}}$]
\item If $\a$ and $\b$ are nonempty subsets of $\g$, then ${\a}^{\hskip -.2mm \b}\;$ is the
centralizer of $\b$ in $\a$,
\begin{equation*}\label{zentralizer}
{\a}^{\hskip -.2mm \b}:=\{x\in
\a\mid [x, y]=0 \text{ for all }
y\in \b\},
\end{equation*}
and $[\a,\b]$ is the $k$-linear span of all $[x, y]$, where $x\in\a$, $y\in\b$.

\item The center  of $\a$, i.e., $\a^{\hskip -.2mm\a}$, is denoted by
$\z(\a)$.

\item For any subset $\s$ in $\g$, we put $\s_i:=\s\cap \g_i$ (the case $\s\cap \g_i=\varnothing$ is not excluded).

\end{enumerate}

The following facts, used below, are proved in \cite{V76}:

There exists a nondegenerate $G$-invariant and $\theta$-invariant scalar multiplica\-tion $\g\times \g\to k$, $(x, y)\mapsto\langle x\,{,}\,y\rangle$.

If $x=x_{\rm s}+x_{\rm n}$ is the Jordan decomposition of an element $x\in \g$ with $x_{\rm s}$ semisimple and $x_{\rm n}$ nilpotent, then $x\in \g_i$ entails $x_{\rm s}, x_{\rm n}\in \g_i$.

There are only finitely many nilpotent $G_0$-orbits in every $\g_i$.

\subsection{}
We first consider the following general construction and introduce the related terminology and notation.

\begin{definition}\label{ccc} Let $M$ be a nonempty set and let $F$ be a nonempty set of functions $M\to k$.\;Define the equivalence relation $\sim_F$ on $M$ by
\begin{equation*}
x\sim_F y \;\; \iff\;\; \mbox{for every $\alpha\in F$, either $\alpha(x)=\alpha(y)=0$ or
$\alpha(x)\alpha(y)\neq 0$.}
\end{equation*}
The equivalence classes of $\sim_F$ are called the {\it cells} of $\sim_F$. The set of all cells of $\sim_F$ is denoted by ${\mathscr C}_F(M)$.
\end{definition}

\begin{proposition}\label{cells} Let $\,V$ be a finite-dimensional vector space over $k$ and let $F$ be a finite subset of the dual space $V^*$.

\begin{enumerate}[\hskip 3.2mm\rm(i)]
\item For every linear subspace $L$ in $V^*$, the set
\begin{equation}\label{CL}
C_L:=\bigg\{v\in V \;\Big\vert\,
\begin{array}{ll}
\alpha(v)=0 &\mbox{for all $\alpha\in F\cap L$,}\\
\beta(v)\neq 0& \mbox{for all $\beta\in F\setminus L$.}
\end{array}
\!\!\bigg\}
\end{equation}
is nonempty.

\item A subset of $\,V$ is a cell of $\sim_F$ if and only if it is $C_L$ for some $L$.

\item The closure $\overline{C_L}$ of $\,C_L$ in $\,V$ is the linear subspace
\begin{equation}\label{closure}
\{v\in V\mid \mbox{$\alpha(v)=0$ for all $\alpha\in F\cap L$}\},
\end{equation}
and the complement of $\,C_L$ in $\overline{C_L}$ is the union of hyperplanes
\begin{equation*}
\textstyle\bigcup_{\beta\in F\setminus L}\{v\in {\overline {C_L}}\mid \beta(v)=0\}.
\end{equation*}
\item The set ${\mathscr C}_F(V)$ is finite.
\end{enumerate}
\end{proposition}
\begin{proof} (i) In view of \eqref{CL}, without changing $C_L$,
we may (and shall) assume that $L$ is the linear span of $L\cap F$.
If $L=V^*$, then $C_L=0$.\;Let $L\neq V^*$; then  the dimension of the
linear subspace \eqref{closure} is positive.\;The restriction of every $\beta\in F\setminus L$
to it is nonzero.\;For, otherwise, ${\rm rk}\big((F\cap L)\cup \beta\big)={\rm rk} (F\cap L)$, whence $\beta$ lies in the linear span of
$F\cap L$, i.e., in $L$,---a contradiction.\;Thus the locus of zeros of the restriction of $\beta$ to the linear subspace \eqref{closure} is its proper linear subspace.\;Whence $C_L$ is nonempty.

(ii) It follows directly from the definitions of $\sim_F$ and $C_L$ that $C_L$ is a cell of $\sim_F$.\;Con\-ver\-sely,
let $C$ be a cell of $\sim_F$ in $V$.\;If $L$  is the linear span of $\{\alpha\in F\mid \alpha|_C=0\}$, then Definition \ref{ccc} implies that every $\beta\in F\setminus L$
vanishes nowhere on $C$.\;Hence $C\subseteq C_L$.\;Since, by (i), $C_L$ is a cell,
this yields $C=C_L$.

(iii) This follows from (i).

(iv) In view of (i), this follows from the finiteness of $F$.
\end{proof}

\begin{corollary}\label{cellvariety}
Every cell of $\sim_F$ in $V$ is an  irreducible smooth rational affine algebraic variety
locally closed in $V$.
\end{corollary}

\subsection{}
We now fix a maximal torus $\t$ of $\g$ such that $\t_1$ is a Cartan subspace of $\g_1$
(since the minimal algebraic subalgebra of  a Cartan subspace is a torus, such a $\t$  exists). Let $R\subset \t^*$ be the root system of $\g$ with respect to $\t$.\;As usual, if $\alpha\in R$, then $\g_\alpha:=\{x\in \g\mid [t, x]=\alpha(t)x\;\mbox{for every $t\in \t$}\}$.

We consider the cells of $\sim_R$ in $\t$.\;
\begin{proposition}\label{cellint} Let $\c$ be a cell of $\sim_R$ in $\t$ and let
\begin{equation*}
R_\c:=\{\alpha\in R\mid \alpha|_\c=0\}.
\end{equation*}
\begin{enumerate}[\hskip 3.2mm\rm (i)]
\item For every nonempty subset $\s\subseteq \c$, the following hold:
\begin{enumerate}
\item[$\rm(i_1)$] $\g^\s$ is a reductive subalgebra of $\,\g$ with the maximal torus $\t$ and the $\t$-root decomposition
    \begin{equation}\label{gs}
    \g^\s=\t\oplus \textstyle\big(\bigoplus_{\alpha\in R^{\ }_{\tiny\mbox{$\c$}}}\g_\alpha\big).
    \end{equation}
\item[$\rm(i_2)$] the center $\z(\g^\s)$ of $\g^\s$ is ${\overline \c}$.
\item[$\rm(i_3)$] the commutator ideal $[\g^\s, \g^\s]$ of $\g^\s$ is ${\overline\c}^\perp\oplus \textstyle\big(\bigoplus_{\alpha\in R^{\ }_{\tiny\mbox{$\c$}}}\g_\alpha\big)$, where ${\overline\c}^\perp$ is the orthogonal complement of $\,\overline\c$ in $\t$ with respect to $\langle\cdot\,{,}\,\cdot\rangle|_\t$\,.
\end{enumerate}
\item If $\,t\in \g$ is a semisimple element such that $\g^t=\g^{\c}$, then
$t\in \c$.
\end{enumerate}
\end{proposition}
\begin{proof} $\rm(i_1)$ Since the elements of $\s$ are semisimple and pairwise commute, the Lie algebra $\g^\s$ is reductive.\;As
$\g^\s=\bigcap_{c\in\s}\g^c$,
to prove \eqref{gs} it suffices to show that   \eqref{gs} holds for $\s$ being a single element $c\in \c$.\;Let $x\in \g$.\;The root decomposition of $\g$ with respect to $\t$,
\begin{equation}\label{rootdecg}
\g=\t\oplus \textstyle\big(\bigoplus_{\alpha\in R}\g_\alpha\big),
\end{equation}
yields $x=t+\sum_{\alpha\in R}x_\alpha$ for some
$t\in \t$, $x_\alpha\in \g_\alpha$.\;Hence $[c,x]=\sum_{\alpha\in R}\alpha(c)x_\alpha$.
Therefore, $x\in \g^c$ if and only if $\alpha(c)=0$ for every $\alpha$ such that $x_\alpha\neq 0$.\;The condition $c\in \c$ implies that $\alpha(c)=0$ if and only if $\alpha\in R_\c$.\;Whence
$x\in \g^c$ if and only if $x_\alpha=0$ for all $\alpha\notin R_\c$.\;This and \eqref{rootdecg} prove (i).

$\rm(i_2)$ Since the center of any reductive Lie algebra lies in every maximal torus of the latter,
$\rm(i_1)$ and Proposition \ref{cells} entail
\begin{align*}\z(\g^\s)&=\{t\in \t\mid \mbox{$[t, \g_\alpha]=\alpha(t)\g_\alpha=0$ for each $\alpha\in R_\c$}\}\\
&=\{t\in \t\mid \mbox{$\alpha(t)=0$ for each $\alpha\in R_\c$}\}=\overline\c.
\end{align*}

$\rm(i_3)$ This follows from \eqref{gs} and $\rm(i_2)$.

(ii) As $\t$ lies in $\g^\c$ by ${\rm (i_1)}$, the equality $\g^t=\g^\c$ entails that $t$ commutes with $\t$.\;As $t$ is semisimple and $\t$ is  a maximal torus,  this shows that $t\in \t$.\;Let $\c'$ be the unique cell in $\t$ containing $t$.\;From ${\rm (i_2)}$ we obtain $\overline{\c'}=\z(\g^t)=\z(\g^\c)=\overline{\c}$.\;By Proposition \ref{cells}(ii),(iii), this
yileds $\c'=\c$.
\end{proof}

\begin{corollary}\label{indep} Every two nonempty subsets of any cell of $\sim_R$ in $\t$
have the same centralizers in $\g$.
\end{corollary}

\subsection{} The decomposition $\t=\bigsqcup_{\c\,\in {\mathscr C}_R(\t\,)}\c$ yields the decomposition
\begin{equation}\label{dect}
\t_1=\textstyle\bigsqcup_{\c\,\in {\mathscr C}_R(\t\,)}\c_1.
\end{equation}
\begin{definition}\label{packet} For every cell $\c\!\in\!  {\mathscr C}_R(\t)$ such that $\c_1\neq\varnothing$
and every nilpotent element
$n\in \g^\c_1$,
the set $G_0\cd (\c_1+n)$
is called a {\it packet} in $\g_1$.
\end{definition}

\begin{remark} {\rm  By \cite[Prop.\;39.1.5]{TY05} (cf.\;also Proposition \ref{disj} below), for $m=1$ in \eqref{gradC}, Definition \ref{packet} is equivalent to the usual definition  of a packet (also known as Jordan/decomposition class) in $\g$;
see  \cite[39.1.3]{TY05} and the references in \cite{P08}.
}
\end{remark}

\subsection{} In this subsection we prove some basic properties of packets in $\g$. Below, for every locally closed subset $X$ in $\g_1$, the notation $X^{\rm reg}$ refers to the action of  $G_0$ on $\g_1$ if otherwise is not stated.

\begin{proposition}\label{cover} The union of all packets in $\g_1$ coincides with $\g_1$.
\end{proposition}
\begin{proof} Let $x\in \g_1$; we should show that $x$ lies in a packet.\;We have $x_{\rm s}, x_{\rm n}\in \g_1$.\;Defi\-ni\-tion  \ref{packet} shows that packets in $\g_1$ are $G_0$-stable.\;On the other hand, since $\t_1$ is a Cartan subspace in $\g_1$ (by the choice of $\t$),  the $G_0$-orbit of $x_{\rm s}$ intersects $\t_1$, see
 \cite[Cor.\;of Thm.\;1]{V76}.\;Hence when proving that $x$ lies in a packet, we can (and shall) assume that $x_{\rm s}\in \t_1$.\;In view of  \eqref{dect}, there exists a cell $\c\in {\mathscr C}_R(\t)$ such that
$x_{\rm s}\in \c_1$.\;By the definition of the Jordan decomposition of $x$, we have $x_{\rm n}\in \g_1^{x_{\rm s}}$.\;In view of Corollary \ref{indep}, we have $\g_1^{x_{\rm s}}=\g_1^\c$.\;Hence, by
Definition \ref{packet}, $G_0\cd (\c_1+x_{\rm n})$ is a packet in $\g_1$; clearly
$x$ lies in it.
\end{proof}

\begin{proposition}\label{pf}
There are only finitely many packets in $\g_1$.
\end{proposition}
\begin{proof} Since the set ${\mathscr C}_R(\t)$ is a finite by Proposition \ref{cells}(iv),
it suffices to show that, for any given cell $\c\in {\mathscr C}_R(\t)$ with $\c_1\neq \varnothing$, there are only finitely many packets of the form $G_0\cd (\c_1+n)$ where $n$ is a nilpotent element of  $\g^\c_1$.

To this end, consider the Lie subalgebra $\g^{\hskip -.1mm \c_1}$ of $\g$.\;By Proposition \ref{cellint}(i), it is reductive.\;As $\g_1$ is a weight space of $\theta$, the inclusion
$\c_1\subseteq \g_1$ entails that $\g^{\hskip -.1mm\c_1}$ is $\theta$-stable, i.e.,  $\g^{\hskip -.1mm \c_1}$ is a graded
subalgebra of the graded Lie algebra $\g$.\;In views of Corollary \ref{indep}, we have
$\g^{\hskip -.1mm \c_1}=\g^{\hskip -.1mm\c}$.

Let $S$ be a closed connected subgroup of $G_0$ such that
\begin{equation}\label{fixx}
{\rm Lie}\,S=\g^{\hskip -.1mm\c}_0.
\end{equation}
By \cite[Prop.\;2]{V76}, there are only finitely many
orbits of the adjoint action of $S$ on the variety of nilpotent elements of  $\g^{\hskip -.1mm\c}_1$.\;Let $n_1,\ldots, n_d$ be the representatives of these orbits.

Now consider a packet $G_0\cd (\c_1+n)$, where $n$ is a nilpotent element of
$\g^{\hskip -.1mm\c}_1$.\;By the aforesaid, $S\cd n=S\cd n_i$ for some $i$.\;Since, in view of \eqref{fixx}, every element of $\c_1$ is a fixed point of $S$, this implies that \begin{equation}\label{S}
S\cd (\c_1+n)=S\cd (\c_1+n_i).
\end{equation}
In turn, since $S$ is a
subgroup of $G_0$, from Definition \ref{packet} and \eqref{S} we obtain
$$G_0\cd (\c_1+n)=G_0S\cd (\c_1+n)\overset{\eqref{S}}{=\hskip -.8mm=}G_0S\cd (\c_1+n_i)=G_0\cd (\c_1+n_i).$$ This proves the required finiteness statement.
\end{proof}

\begin{proposition}\label{disj} For any $x, y\in \g_1$ the following properties are equivalent:
\begin{enumerate}[\hskip 3.2mm\rm(i)]

\item[\rm(P)] There exists a packet $G_0\cd (\c_1+n)$ in $\g_1$
containing both
$x$ and $y$.
\item[\rm(J)] There exists an element $g\in G_0$ such that
\begin{equation}\label{JJJ}
g\cd \g^{x_{\rm s}}=\g^{g\cdot x_{\rm s}}=\g^{y_{\rm s}}\quad\mbox{and}\quad
g\cd x_{\rm n}=y_{\rm n}.
\end{equation}
\end{enumerate}
\end{proposition}\label{equi}
\begin{proof} (P)$\Rightarrow$(J) Let $x, y\in G_0\cd (\c_1+n)$.
By Definition \ref{packet}, there are $c_1, c_2\in \c_1$ and $g_1, g_2\in G_0$ such that $g_1\cd x=c_1+n$, $g_2\cd y=c_2+n$. Hence
\begin{equation}\label{sn}
\begin{split}
g_1\cd x_{\rm s}&=(g_1\cd x)_{\rm s}=c_1,\quad g_1\cd x_{\rm n}=(g_1\cd x)_{\rm n}=n,\\
g_2\cd y_{\rm s}&=(g_2\cd y)_{\rm s}=c_2,\quad\hskip .4mm g_2\cd y_{\rm n}=(g_2\cd y)_{\rm n}=n,
\end{split}
\end{equation}
From \eqref{sn} and Corollary \ref{indep}  we obtain $\g^{g_1\cdot x_{\rm s}}=\g^{g_2\cdot y_{\rm s}}$ and $g_1\cd x_{\rm n}=g_2\cd y_{\rm n}$; whence \eqref{JJJ} with $g=g_2^{-1}g_1$.

(J)$\Rightarrow$(P) Assume that \eqref{JJJ} holds for $x, y\in \g_1$.\;By Proposition \ref{cover}, there is a packet $G_0\cd (\c_1+n)$
in $\g_1$ containing $y$; we have to show that $x\in G_0\cd (\c_1+n)$.\;
By Definition \ref{packet}, there are $c\in \c_1$ and $h\in G_0$ such that $y=h\cd (c+n)$; whence
$y_{\rm s}=h\cd c$, $y_{\rm n}=h\cd n$.\;Plugging this in \eqref{JJJ}, we obtain
\begin{equation}\label{interm}
\g^{w\cdot x_{\rm s}}=\g^c,\;\; w\cd x_{\rm n}=n,\quad \mbox{where $w=h^{-1}g$}.
\end{equation}

By Proposition \ref{cellint}$\rm(i_1)$, we have $\t\subset \g^c$.\;Hence, by the first equality in \eqref{interm}, the semisimple element $w(x_{\rm s})$ centralizes the maximal torus $\t$.\;Whence $w\cd x_{\rm s}\in \t$ and, therefore, there is a unique cell $\c'$ in $\t$ such that \begin{equation}\label{c'}
w\cd x_{\rm s}\in \c'.
\end{equation}

From Proposition \ref{cellint}$\rm(i_2)$  we then infer that $\z(\g^{w\cdot  x_{\rm s}})=\overline{\c'}$. On the other hand, by the same reason,
$\z(\g^c)=\overline{\c}$.\;Combining this with the first equality in \eqref{interm}, we conclude  that $\overline{\c'}=\overline{\c}$.\;By Proposition \ref{cells}(iii), this yields
\begin{equation}\label{c'c}
\c'=\c.
\end{equation}

From \eqref{c'}, \eqref{c'c}, the second equality in \eqref{interm}, and Definition \ref{packet} we now conclude that $x\in G_0\cd (\c_1+n)$.
 \end{proof}

\begin{corollary} The packets in $\g_1$ do not depend on the choice of  a torus $\,\t$.
\end{corollary}

\begin{corollary}\label{disjo}
 Every two packets in $\g_1$
are either equal
or disjoint.
\end{corollary}

\begin{proof} This is because property (J)  is clearly an equivalence relation on $\g_1$ and, by Proposition \ref{disj}, packets in $\g_1$
 are precisely its equi\-va\-lence classes.
\end{proof}

\begin{proposition}\label{propac} For every packet
$G_0\cd (\c_1+n)$ 
and its closure $\overline{G_0\cd(\c_1+n)}$ in\;$\g_1$,
the following hold:
\begin{enumerate}[\hskip 3.2mm\rm(i)]

\item $\g^x_i=\g^\c_i\cap \g_i^n$ for any  $x\in \c_1+n$ and $i$;

\item  $[\g_0, x]=[\g_0, y]$ for any $x, y\in \c_1+n$.

\item  $G_0\cd(\c_1\!+n)$ is irreducible
and contains a dense open subset of
$\overline{G_0\cd(\c_1\!+\!n)}$.

    \item $G_0\cd(\c_1\!+\!n)$ is contained in a sheet of $\g_1$.

\item
$\overline{G_0\cd (\c_1 \!+\! n)}$
is an irreducible algebraic variety of dimension $d+\dim \c_1$,
where $d$ is the dimension of $\,G_0$-orbits in a sheet containing $G_0\cd (\c_1 \!+n)$.
\end{enumerate}
\end{proposition}
\begin{proof}
If $c \in \c_1$ and $x=c+n$, then $c=x_{\rm s}$ and $n=x_{\rm n}$.\;This, the uniqueness of the Jordan decomposition, and Corollary \ref{indep} imply that
\begin{equation}\label{stabi}
 \g^x=\g^c\cap \g^n=\g^\c\cap \g^n.
 \end{equation}
  Since $c, n, x$ are the homogeneous elements, $\g^c$, $\g^n$, $\g^x$ are the
  graded subalgeb\-ras of $\g$.\;Taking the $i$th components, from \eqref{stabi} we obtain
   (i).

   The subspace $\g_{-1}$ is dual to $\g_1$ with respect to $\langle\cdot\,{,}\,\cdot\rangle$, and the orthogonal
complement in $\g_1$ to $\g_{-1}^x$ with respect to
this duality is $[\g_0, x]$; see \cite[Prop.\;5]{V76}.\;Since $\g_{-1}^x$, by (i), is the same for
all $x\in \c_1+n$, this proves (ii).

Corollary \ref{cellvariety} and the connectedness of $G_0$
yield that $G_0\times \c_1$ is an irreducible smooth affine algebraic variety.\;In view of
 Definition \ref{packet},
the packet $G_0\cd (\c_1+n)$ is the image of the 
morphism:
\begin{equation}\label{phi}
\varphi\colon G_0\times \c_1\to \g_1,\quad (g, c)\mapsto g\cd (c+n).
\end{equation}
By the general properties of morphisms,   this implies (iii).

In view of  \eqref{G0}, we have $\dim G_0\cd x=\dim \g_0-\dim \g_0^x$.\;This and (i) show that $\dim G_0\cd x$ is the same for all $x\in G_0\cd (\c_1+n)$, which, by (iii), means that
$G_0\cd (\c_1+n)\subseteq \overline{G_0\cd (\c_1+n)}^{\rm reg}$. This implies (iv).

We may (and shall) consider $\varphi$ as a dominant morphism $G_0\times \c_1\!\to\! \overline{G_0\!\cd\! (\c_1\!+\!n)}$. By \cite[Lem.\;10.5]{H77},
there is a point $z=(g, c)\in G_0\times \c_1$ such that $\varphi(z)$ is a smooth point of
$\overline{G_0\cd (\c_1+n)}$ and $d\varphi_z$ is a surjective map of the tangent spaces; this implies that
$\dim \overline{G_0\cd (\c_1+n)}$ is equal to the dimension of the image of   $d\varphi_z$.\;We shall now compute this latter dimension.

First note that since $\varphi $ is $G_0$-equivariant, we may (and shall) assume that $g$ is the identity element.\;From
Proposition \ref{cells}(ii) we infer that the closure $\overline{\c_1}$ of
$\c_1$ in $\t_1$ is a linear subspace, and $\overline{\c_1}\setminus \c_1$ is a union of finitely many hyperplanes.\;This shows that $\g_0\oplus \overline{\c_1}$ is
the tangent space of $G_0\times \c_1$ at $z$.\;Given this, we deduce from \eqref{G0} and \eqref{phi} that
\begin{equation}\label{ddd}
\mbox{the image of $d\varphi_z$ is $[\g_0, c+n]+\overline{\c_1}$.}
\end{equation}

For the point $x:=c+n\in \c_1+n$, we have $x_{\rm s}=c$, $x_{\rm n}=n$.\;Therefore,
\begin{equation}\label{cap}
[\g, x]\cap \z(\g^{\hskip -.1mm c})=0;
\end{equation}
see \cite[Lem.\;39.2.8]{TY05}.\;But
$\z(\g^{\hskip -.1mm c})=\overline \c$ by Proposition \ref{cellint}$\rm (i_2)$.\;This and \eqref{cap} yield $[\g, x]\cap \overline \c=0$; whence
\begin{equation}\label{caps}
[\g_0, x]\cap \overline {\c_1}=0.
\end{equation}

Combining \eqref{ddd}, \eqref{caps}, (ii), and taking into account that
$[\g_0, x]$ is the tangent space to $G_0\cd x$ at $x$,
 we now obtain that the dimension of the image of $d\varphi_z$ is equal to $\dim [\g_0, x]+\dim \overline {\c_1}=d+\dim \c_1$.\;This completes the proof of (v).
\end{proof}

\begin{corollary}\label{mopa}
${\rm mod}(G_0:\overline{G_0\cd (\c_1+n)})
=\dim\c_1$.
\end{corollary}

\begin{lemma}\label{xreg}
$x\in (\z(\g^x)_1)^{\rm reg}$ for every $x\in \g_1$.
\end{lemma}
\begin{proof}
If $y\in \z(\g^x)$, then $\g^x\subseteq \g^y$; see \cite[35.3.2]{TY05}.\;As $x$ is homogeneous, $\g^x$ is graded.\;If $y$ is homogeneous, $\g^y$ is graded as well, hence the specified inclusion yields $\g^x_i\subseteq \g^y_i$ for every $i$.\;In particular,
$\g_0^x\subseteq \g_0^y$ for every $y\in \z(\g^x)_1$, whence $\dim G_0\cd y\leqslant \dim G_0\cd x$ and, by
\eqref{regreg}, the claim.
\end{proof}

\begin{lemma}\label{nilpp}
Let $x\in \g_1$ be a nilpotent element.\;Then
$$G_0\cd x\cap \z(\g^x)_1=(\z(\g^x)_1)^{\rm reg}.$$
\end{lemma}
\begin{proof} By by \eqref{regreg} and Lemma \ref{xreg}, for every $z\in \z(\g^x)_1$, we have
\begin{equation}\label{dime}
z\in (\z(\g^x)_1)^{\rm reg} \iff \dim G_0\cd z=\dim G_0\cd x.
\end{equation}

As $x$ is nilpotent and $\g$ is semisimple, all elements of $\z(\g^x)$ are nilpotent (see
\cite[35.1.2]{TY05}).\;Therefore $\z(\g^x)_1$ lies in the variety of nilpotent elements of  $\g_1$.\;By \cite[Prop.\;2]{V76}, this variety is the union of finitely many $G_0$-orbits.\;From this we infer that
among them there is an orbit $\mathcal O$ such that  $\mathcal O\cap \z(\g^x)_1$ is a dense open subset of $\z(\g^x)_1$.\;This and \eqref{dime} entail that
\begin{align}\label{dimO=}
\overline{\mathcal O\cap \z(\g^x)_1}&=\z(\g^x)_1,\\
\mathcal O\cap \z(\g^x)_1&\subseteq (\z(\g^x)_1)^{\rm reg},\label{dimOsub}\\
\dim \mathcal O&=\dim G_0\cd x.\label{dim=dim}
\end{align}

If $z\in (\z(\g^x)_1)^{\rm reg }$, then  \eqref{dime}, \eqref{dim=dim} yield $\dim G_0\cd z=\dim \mathcal O$. Since, by \eqref{dimO=}, $z\in\overline{\mathcal O}$, the latter equality implies that
$G_0\cd z =\mathcal O$.\;This equality and \eqref{dimOsub} yield
\begin{equation}\label{===}
\mathcal O\cap \z(\g^x)_1= (\z(\g^x)_1)^{\rm reg}.
\end{equation}

From \eqref{===} and Lemma \ref{xreg} we now infer that $\mathcal O=G_0\cd x$.\;This equality and \eqref{===} complete the proof.
\end{proof}

\begin{proposition}\label{clapac}
For all $x, y\in \g_1$, the following conditions are equi\-va\-lent:
\begin{enumerate}[\hskip 4.2mm\rm(i)]
\item $x$ and $y$ lie in the same packet in $\g_1$.
\item There exists $g\in G_0$ such that $\g^y=g\cd {\g^x}$.
\item There exists $g\in G_0$ such that $\z(\g^y)=g\cd \z(\g^x)$.
\end{enumerate}
\end{proposition}
\begin{proof} The implication (ii)$\Rightarrow$(iii) is clear.

Assume that (iii) holds.\;Then $\z(\g^y)=\z(\g^{g\cdot x})$, which implies that
$\g^y=\g^{g\cdot x}=g\cd \g^x$; see \cite[35.3.2]{TY05}. This proves
the implication (iii)$\Rightarrow$(ii).

Assume that (i) holds.\;Then by Proposition \ref{disj} there exists $g\in G_0$ such that
\eqref{JJJ} holds.\;Therefore
\begin{equation*}
\g^y=\g^{y_{\rm s}}\cap \g^{y_{\rm s}}=g\cd \g^{x_{\rm s}}\cap g\cd \g^{x_{\rm n}}=
g\cd (\g^{x_{\rm s}}\cap \g^{x_{\rm n}})=g\cd \g^x.
\end{equation*}
This proves
the implication (i)$\Rightarrow$(ii).

Finally, let us prove that $x$ and $y$ lie same packet in $\g_1$  assuming that (iii) holds.
Since packets are $G_0$-stable, we may (and shall) assume that $g$ is the identity element,
i.e.,
\begin{equation}\label{g=e}
\z(\g^x)=\z(\g^y).
\end{equation}

 According to \cite[39.1.1]{TY05},
 \begin{equation}\label{z+z}
 \begin{split}
 \z(\g^x)&=\z(\g^{x_{\rm s}})\oplus \z\big([\g^{x_{\rm s}}, \g^{x_{\rm s}}]^{x_{\rm n}}\big),\\
 \z(\g^y)&=\z(\g^{y_{\rm s}})\oplus \z\big([\g^{y_{\rm s}}, \g^{y_{\rm s}}]^{y_{\rm n}}\big),
 \end{split}
 \end{equation}
 and the first (respectively, second) direct summand in each of the equalities \eqref{z+z} is the set of all semisimple (respectively, nilpotent) elements of the left-hand side of this equality.\;This and \eqref{g=e} then entail
 \begin{equation}\label{ceco}
 \begin{split}
 \z(\g^{x_{\rm s}})&=\z(\g^{y_{\rm s}}),\\
 \z\big([\g^{x_{\rm s}}, \g^{x_{\rm s}}]^{x_{\rm n}}\big)&=\z\big([\g^{y_{\rm s}}, \g^{y_{\rm s}}]^{y_{\rm n}}\big).
 \end{split}
 \end{equation}

 Since $x$ and $y$ are homogeneous elements, the left- and right-hand sides of \eqref{ceco} are
 the graded subalgebras of $\g$.\;The subalgebras
 $[\g^{x_{\rm s}}, \g^{x_{\rm s}}]$, $[\g^{y_{\rm s}}, \g^{y_{\rm s}}]$ are semisimple and $x_{\rm n}\in [\g^{x_{\rm s}}, \g^{x_{\rm s}}]_1$, $y_{\rm n}\in [\g^{y_{\rm s}}, \g^{y_{\rm s}}]_1$.

 Let $H$ be the connected closed sub\-group of $G_0$ such that ${\rm Lie}\,H=
 [\g^{x_{\rm s}}, \g^{x_{\rm s}}]_0$. Applying Lemma \ref{nilpp}, in which $x$ and $\g$ are replaced respectively by $x_{\rm n}$ and $[\g^{x_{\rm s}}, \g^{x_{\rm s}}]$, and
 taking \eqref{ceco} into account, we obtain:
 \begin{equation}\label{rrrr}
 H\cd x_{\rm n}\cap \z\big([\g^{x_{\rm s}}, \g^{x_{\rm s}}]^{x_{\rm n}}\big)_1=
 \big(\z\big([\g^{x_{\rm s}}, \g^{x_{\rm s}}]^{x_{\rm n}}\big)_1\big)^{\rm reg}
 \overset{\eqref{ceco}}{=\hskip -1mm=}
 \big(\z\big([\g^{y_{\rm s}}, \g^{y_{\rm s}}]^{y_{\rm n}}\big)_1\big)^{\rm reg},
 \end{equation}
 where the notation $X^{\rm reg}$ refers to the action of
 $H$ on $[\g^{x_{\rm s}}, \g^{x_{\rm s}}]_1$.

 By Lemma \ref{xreg}, the element $y_{\rm n}$ lies in the right-hand side of \eqref{rrrr}.\;Therefore \eqref{rrrr} implies that there exists $g\in H$ such that
 \begin{equation}\label{ginH}
 g\cd x_{\rm n}=y_{\rm n}.
 \end{equation}

 Since $[\g^{x_{\rm s}}, \g^{x_{\rm s}}]$ and $\z(\g^{x_{\rm s}})$ commute, $H$ acts trivially on
 $\z(\g^{x_{\rm s}})$.\;This and the first equality in \eqref{ceco} yield
 \begin{equation}\label{paccc}
\z(\g^{g\cdot x_{\rm s}})= g\cd \z(\g^{x_{\rm s}})=\z(\g^{x_{\rm s}})\overset{\eqref{ceco}}{=\hskip -1mm=}\z(\g^{y_{\rm s}}).
 \end{equation}
From \eqref{paccc}, using \cite[35.3.2]{TY05},  we infer that
\begin{equation}\label{fi}
\g^{g\cdot x_{\rm s}}=\g^{y_{\rm s}}.
\end{equation}

By Proposition \ref{disj}, the equalities \eqref{ginH} and \eqref{fi} imply (i), thereby
comp\-leting the proof of implication (iii)$\Rightarrow$(i).
\end{proof}

Using Proposition \ref{clapac}, we now can show that in fact a property stronger than Proposition \ref{propac}(iii) holds:

\begin{proposition}\label{lc} Every packet $P$ in $\g_1$ is open in its closure in $\g_1$.
\end{proposition}

\begin{proof} By Proposition \ref{propac}, there exists an integer $d$ such that $P$
lies in the locally closed subset $X=\{x\in \g\mid \dim \g^x=d\}$ of $\g$.\;Let ${\rm Grass}(\g, d)$ be the Grassmannian of $d$-dimensional linear subspaces in $\g$ endowed with the natural action of $G$.\;The map
\begin{equation*}
\varphi\colon X\to {\rm Grass}(\g, d),\quad x\mapsto \g^x.
\end{equation*}
is $G$-equivariant.\;Let $x\in P$. Then Proposition \ref{clapac} implies that
\begin{equation}\label{PG0}
P=\g_1\cap \varphi^{-1}(G_0\cd \varphi(x)).
\end{equation}

Since $\varphi$
is a morphism  (see \cite[19.7.6 and 29.3.1]{TY05}) and orbits of algebraic transformation groups are open in their closures, \eqref{PG0} shows that $P$ enjoys the latter property
 as well.
\end{proof}

\subsection{} We conclude with describing the relationship between the sheets and the pa\-ckets in $\g_1$.
Below bar stands for the closure in $\g_1$.

\begin{proposition}\label{shee} For every sheet $S$  of $\,\g_1$, there is a unique packet
$P$ in $\g_1$ such that $P\subseteq S$ and $\overline P=\overline S$.\;Moreover,
$S={\overline P}^{\rm reg}$.
\end{proposition}

\begin{proof} If $d$ is the dimension of $G_0$-orbits in $S$, then  by Propositions \ref{cover}, \ref{propac} and in view of the definition of a sheet, we have
$$\{x\in \g_1\mid \dim G_0(x)=d\}=\textstyle \bigcup_{i=1}^{p} S_i=\bigcup_{j=i}^{q}P_j,$$
where $S_1,\ldots, S_p$  and $P_1,\ldots, P_q$ are, respectively, some pairwise different sheets and
packets in $\g_1$; the sheet $S$ is one of $S_1,\ldots, S_m$.
Whence
\begin{equation}\label{SPcl}
\textstyle \bigcup_{i=1}^{p} \overline{S_i}=\bigcup_{j=i}^{q}\overline{P_j}.
\end{equation}

As sheets and packets contain open subsets of their closures, we have $\overline {S_i} \nsubseteq \overline{S_j}$ and $\overline {P_i}  \nsubseteq \overline{P_j}$ for all $i\neq j$.
Since all $\overline {S_i}$ and  $\overline{P_j}$ are irreducible,
\eqref{SPcl} then implies that $p=q$ and there is a permutation $\sigma$ of $\{1,\ldots, p\}$ such that $\overline {S_i}=\overline {P_{\sigma(i)}}$.

The last claim of this proposition follows from the definition of a regular sheet.
\end{proof}

\end{document}